\documentclass[12pt,a4paper]{article}
\usepackage{graphicx}
\usepackage{amsfonts}
\usepackage{amsmath}
\usepackage{amssymb}
\usepackage{amsthm}
\usepackage[ansinew]{inputenc}
\usepackage{latexsym}
\usepackage{tabularx}
\usepackage{color}
\usepackage{enumitem}
\usepackage[active]{srcltx}
\allowdisplaybreaks
\newtheorem{thm}{Theorem}

\newtheorem{lem}{Lemma}

\newtheorem{rem}{Remark}

\newtheorem{algorithm}{Algorithm}

\newcommand{\norm}[1]{\left\Vert#1\right\Vert}
\newcommand{\abs}[1]{\left\vert#1\right\vert}

\newcommand{\To}{\rightarrow}

\newcommand{\bsgamma}{\boldsymbol{\gamma}}

\newcommand{\bsx}{\boldsymbol{x}}

\newcommand{\bsp}{\boldsymbol{p}}

\newcommand{\bsm}{\boldsymbol{m}}

\newcommand{\bsy}{\boldsymbol{y}}
\newcommand{\bssigma}{\boldsymbol{\sigma}}

\newcommand{\cH}{{\cal H}}

\newcommand{\bsone}{\boldsymbol{1}}
\newcommand{\rd}{\,\mathrm{d}}

\newcommand{\NN}{\mathbb{N}}
\newcommand{\ZZ}{\mathbb{Z}}
\newcommand{\RR}{\mathbb{R}}
\newcommand{\PP}{\mathbb{P}}

\newcommand{\QQ}{\mathbb{Q}}

\newcommand{\uu}{\mathfrak{u}}

\newcommand{\simi}{\mathrm{mid}}
\newcommand{\simp}{\mathrm{smp}}

\newcommand{\EXCLUDE}[1]{}

\makeatletter
\newcommand{\rdots}{\mathinner{\mkern1mu\lower-1\p@\vbox{\kern7\p@\hbox{.}}
\mkern2mu \raise4\p@\hbox{.}\mkern2mu\raise7\p@\hbox{.}\mkern1mu}}
\makeatother
\begin{document}
\title{Component-by-component construction of shifted Halton sequences}

\author{Peter Kritzer and Friedrich Pillichshammer\thanks{The authors are supported by the
Austrian Science Fund (FWF): Projects F5506-N26 (Kritzer)  and
F5509-N26 (Pillichshammer), respectively, which are part of the
Special Research Program "Quasi-Monte Carlo Methods: Theory and
Applications".}}

\date{}

\maketitle

\begin{center}
Dedicated to H. Niederreiter on the occasion of his 70th
birthday.
\end{center}

\begin{abstract}
We study quasi-Monte Carlo integration in a weighted anchored
Sobolev space. As the underlying integration nodes we consider Halton
sequences in prime bases $\bsp=(p_1,\ldots,p_s)$ which are shifted
with a $\bsp$-adic shift based on $\bsp$-adic arithmetic.
The error is studied in the worst-case setting. In a recent paper,
Hellekalek together with the authors of this article proved
optimal error bounds in the root mean square sense, where the mean was extended over
the uncountable set of all possible $\bsp$-adic shifts. Here we show that
candidates for good shifts can in fact be chosen from a finite set and can be
found by a component-by-component algorithm.
\end{abstract}

\noindent\textbf{Keywords:} Quasi-Monte Carlo integration, shifted
Halton sequences, worst-case error.

\noindent\textbf{MSC:} 65D30, 65C05, 11K38, 11K45.

\section{Introduction}

We study the problem of approximating the value of the integral
$I_s(f):=\int_{[0,1]^s}f(\bsx)\rd\bsx$ of functions $f$ belonging
to a reproducing kernel Hilbert space $\cH(K)$ of functions
$[0,1]^s\rightarrow\RR$. One way of numerically approximating
$I_s(f)$ is to employ a quasi-Monte Carlo (QMC) rule,
\[Q_{N,s}(f):=\frac{1}{N}\sum_{n=0}^{N-1}f(\bsx_n),\]
where $P_{N,s}=\{\bsx_0,\bsx_1,\ldots,\bsx_{N-1}\}$ is a set of
$N$ deterministically chosen points in $[0,1)^s$. It is well known
(see, e.g., \cite{DP10,DT97,KN74,N92,SJ94}) that point sets which
are in some way evenly distributed in the unit cube yield a low
integration error when applying a QMC rule for approximating
$I_s(f)$.

We study the error of QMC rules in the worst-case setting. The
worst-case error of an algorithm $Q_{N,s}$ based on nodes
$P_{N,s}$ is defined as the worst integration error over the
unit-ball of $\cH(K)$, i.e., $$e_{N,s}(P_{N,s},K)=\sup_{f \in
\cH(K) \atop \|f\|_{K} \le 1}|I_s(f)-Q_{N,s}(f)|.$$ An essential
question in the theory of QMC methods is how the sample nodes
$P_{N,s}$ of a QMC rule $Q_{N,s}$ should be chosen.

\paragraph{Shifted Halton sequences.}
In this paper we focus on a special kind of point sequences
underlying a QMC rule, namely Halton sequences (cf.~\cite{H60})
whose definition is based on the radical inverse function. Let $p
\ge 2$ be an integer, $\NN=\{1,2,3,\ldots\}$, and $\NN_0=\NN \cup
\{0\}$. For $n\in\NN_0$, let $n=n_0 + n_1 p + n_2 p^2+\cdots$ be
the base $p$ expansion of $n$ (which is of course finite) with
digits $n_i\in\{0,1,\ldots,p-1\}$ for $i\ge 0$. The radical
inverse function $\phi_p:\NN_0\To [0,1)$ in base $p$ is defined by
\[\phi_p (n):=\sum_{r=0}^{\infty} \frac{n_{r}}{p^{r+1}}.\]
Halton sequences can be defined for any dimension
$s\in\NN$. Let $p_1,\ldots,p_s \ge 2$ be $s$ integers, and let
$\bsp=(p_1,\ldots,p_s)$. Then the $s$-dimensional Halton sequence
$H_{\bsp}$ in bases $p_1,\ldots,p_s$ is defined to be the sequence
$H_{\bsp}=(\bsx_n)_{n\ge 0}\subseteq [0,1)^s$, where
\[\bsx_n=(\phi_{p_1}(n),\phi_{p_2}(n),\ldots,\phi_{p_s}(n)),\ \ \ \mbox{ for } \ n\in \NN_0.\]
It is well known (see, e.g.,~\cite{DP10,N92}) that Halton
sequences have good distribution properties if and only if the
bases $p_1,\ldots,p_s$ are mutually relatively prime, and for the
sake of simplicity we assume throughout the rest of the paper that
$\bsp=(p_1,\ldots,p_s)$ consists of $s$ mutually different prime
numbers.

We also introduce a method of randomizing the elements of the
Halton sequence which is referred to as a $\bsp$-adic shift. This
special case of randomization is based on arithmetic over the
$p$-adic numbers and is perfectly suited for Halton sequences
$H_{\bsp}$.

Let $p$ be a prime number. We define the set of $p$-adic numbers
as the set of formal sums
\begin{equation*}
\mathbb{Z}_p = \left\{z = \sum_{r=0}^\infty z_r p^r\, : \, z_r \in
\{0,1,\ldots,p-1\} \mbox{ for all } r \in \NN_0\right\}.
\end{equation*}
Clearly $\mathbb{N}_0 \subseteq \mathbb{Z}_p$. For two nonnegative
integers $y, z \in \mathbb{N}_0 \subseteq \mathbb{Z}_p$, the sum
$y+z \in \mathbb{Z}_p$ is defined as the usual sum of integers.
The addition can be extended to all $p$-adic numbers. The set
$\mathbb{Z}_p$ with this addition, which we denote by $+_{\ZZ_p}$,
then forms an abelian group.

As an extension of the radical inverse function defined above, we
define the so-called Monna map
\begin{equation*}
\phi_p:\mathbb{Z}_p \to [0,1)\ \ \mbox{by}\ \ \phi_p(z): =
\sum_{r=0}^\infty \frac{z_r}{p^{r+1}} \pmod{1}
\end{equation*}
whose restriction to $\NN_0$ is exactly the radical inverse
function in base $p$. In order to keep the used notation at a
minimum we denote both, the Monna map and the radical inverse
function, by $\phi_p$. We also define the inverse
\begin{equation*}
\phi_p^+: [0,1)\to \mathbb{Z}_p\ \ \mbox{by}\ \
\phi_p^+\left(\sum_{r=0}^\infty \frac{x_r}{p^{r+1}}\right) :=
\sum_{r=0}^\infty x_r p^r,
\end{equation*}
where we always use the finite $p$-adic representation for
$p$-adic rationals in $[0,1)$. By a $p$-adic rational, we understand a number 
in $[0,1)$ that can be represented by a finite $p$-adic expansion. 

For a prime number $p$ and for $x \in [0,1)$ we consider the
following $p$-adic shifts:
\begin{itemize}
\item {\bf $p$-adic shift:} for $\sigma \in [0,1)$, we define
$x\oplus_{p}\sigma\in [0,1)$ to be
\[x\oplus_{p}\sigma=\phi_{p} (\phi_{p}^+ (x) +_{\ZZ_p}
\phi_{p}^+ (\sigma)).\] 
\item {\bf simplified $p$-adic shift:} for
$m\in\NN$ and $\sigma \in [0,1)$, we write $x\oplus_{p,m}^{\simp}
\sigma$ to be the truncation of $x\oplus_p \sigma$ to the
$m$ most significant digits, i.e., if $\phi_{p}^+ (x) +_{\ZZ_p}
\phi_{p}^+ \sigma =\sum_{r=1}^\infty y_r p^{r-1} \in \ZZ_p$, then
$$x\oplus_{p,m}^{\simp} \sigma= \phi_p\left(\sum_{r=1}^m y_r
p^{r-1}\right).$$ 
\item {\bf mid-simplified $p$-adic shift:} for
$m\in\NN$ and $\sigma \in [0,1)$, we write $$x
\oplus_{p,m}^{\simi} \sigma  = (x\oplus_{p,m}^{\simp} \sigma)+
\frac{1}{2p^m}.$$
\end{itemize}
If the choice of $m$ is clear from the context, we may often omit
$m$ in the notation $\oplus_{p,m}^{\simp}$ and
$\oplus_{p,m}^{\simi}$ and write $\oplus_{p}^{\simp}$ and
$\oplus_{p}^{\simi}$ instead.

In the $s$-variate case, for given bases $\bsp=(p_1,\ldots,p_s)$,
a point $\bsx=(x_1,\ldots,x_s)\in [0,1)^s$, and given
$\bssigma=(\sigma_1,\ldots,\sigma_s)\in [0,1)^s$ and
$\bsm=(m_1,\ldots,m_s) \in \NN^s$, the above shifts are defined
component-wise and we write $\bsx\oplus_{\bsp}\bssigma\in
[0,1)^s$, $\bsx\oplus_{\bsp,\bsm}^{\simp} \bssigma$ and $\bsx
\oplus_{\bsp,\bsm}^{\simi} \bssigma$, respectively.

For a point set $Y=\{\bsy_n \ : \ n=0,\ldots ,N-1\}$ we write
$$Y\oplus \bssigma:=\{\bsy_n \oplus \bssigma \ : \ n=0,\ldots
,N-1\}\ \ \mbox{ where } \ \oplus \mbox{ is either }
\oplus_{\bsp},\oplus_{\bsp,\bsm}^{\simp},\mbox{ or }\oplus_{\bsp,\bsm}^{\simi}.$$

\paragraph{A weighted Sobolev space.} In this paper, we are going to consider the problem of numerical
integration of functions $f$ that belong to a weighted anchored
Sobolev space. Before we give the definition we introduce some
notation which we require for the following: assume that
$\bsgamma=(\gamma_j)_{j=1}^{\infty}$ is a non-increasing sequence
of positive weights, where $1\ge \gamma_1\ge\gamma_2\ge \cdots$.
These weights are used in order to model the influence of the
different variables of the integrands, an idea which was
introduced by Sloan and Wo\'{z}niakowski \cite{SW98}. For $s \in
\NN$ let $[s]:=\{1,\ldots,s\}$. For $\uu\subseteq [s]$,
$\bsx_{\uu}$ denotes the projection of $\bsx\in [0,1]^s$ onto
$[0,1]^{|\uu|}$ consisting of the components whose indices are
contained in $\uu$. Furthermore we write $(\bsx_{\uu},\bsone) \in
[0,1]^s$ for the point where those components of $\bsx$ whose
indices are not in $\uu$ are replaced by 1.

We consider a weighted anchored Sobolev space $\cH
(K_{s,\bsgamma})$ with anchor $\bsone=(1,1,\ldots,1)$ consisting
of functions on $[0,1]^s$ whose first mixed partial derivatives
are square integrable. This space is a reproducing kernel Hilbert
space with kernel function
\begin{equation}\label{ker_sob_space}
K_{s,\bsgamma}(\bsx,\bsy)=\prod_{j=1}^s (1+\gamma_j \min
(1-x_j,1-y_j)) \ \ \mbox{ for } \bsx,\bsy \in [0,1]^s,
\end{equation}
where $\bsx=(x_1,x_2,\ldots,x_s)$ and $\bsy=(y_1,y_2,\ldots,y_s)$.
The inner product is given by
\[\langle f,g \rangle_{K_{s,\bsgamma}}=\sum_{\uu\subseteq [s]}\gamma_{\uu}^{-1} \int_{[0,1]^{\abs{\uu}}}
  \frac{\partial^{\abs{\uu}}}{\partial \bsx_{\uu}} f(\bsx_{\uu},\bsone)\frac{\partial^{\abs{\uu}}}{\partial \bsx_{\uu}}
g(\bsx_{\uu},\bsone)\rd \bsx_{\uu}.\] Here
$\gamma_{\uu}=\prod_{j\in\uu}\gamma_j$; in particular
$\gamma_{\emptyset}=1$. Furthermore, we denote by
$\frac{\partial^{\abs{\uu}}}{\partial \bsx_{\uu}} h$ the
derivative of a function $h$ with respect to the $x_j$ with
$j\in\uu$. The norm in $\cH (K_{s,\bsgamma})$ is given by
$\norm{f}_{K_{s,\bsgamma}}=\sqrt{\langle
f,f\rangle_{K_{s,\bsgamma}}}$. The Sobolev space $\cH
(K_{s,\bsgamma})$ has been studied frequently in the literature
(see, among many references,
e.g.~\cite{DKPS05, DP05, HKP14, KP11, K03, NW08, SW98, wang}).\\

It is well known that the squared worst-case integration error in
a reproducing kernel Hilbert space can be expressed in terms of
the kernel function. In the particular case of the kernel
$K_{s,\bsgamma}$, it is easily derived with the help of
\cite[Proposition~2.11]{DP10} that for
$P_{N,s}=\{\bsx_0,\ldots,\bsx_{N-1}\}$ in $[0,1)^s$, where
$\bsx_n=(x_{n,1},\ldots,x_{n,s})$ for $n=0,1,\ldots,N-1$, we have
\begin{eqnarray}\label{fo_wc_err}
e^2_{N,s}(P_{N,s},K_{s,\bsgamma}) & = &
\prod_{i=1}^s\left(1+\frac{\gamma_i}{3}\right)-\frac{2}{N}
\sum_{n=0}^{N-1}\prod_{i=1}^s\left(1+\frac{\gamma_i}{2}(1-x_{n,i}^2)\right)\nonumber\\
&& \mbox{}+\frac{1}{N^2}\sum_{n,h=0}^{N-1}\prod_{i=1}^s
\left(1+\gamma_i\min(1-x_{n,i},1-x_{h,i})\right).
\end{eqnarray}
Hence the worst-case error can be computed at a cost of $O(s N^2)$ arithmetic operations.\\

In \cite{HKP14} the authors studied the root mean square
worst-case error in $\cH(K_{s,\bsgamma})$ of the $\bsp$-adic
shifted Halton sequence extended over all $\bsp$-adic shifts,
i.e.,
\[\widehat{e}_{N,s}(H_{\bsp},K_{s,\bsgamma}):=\sqrt{\mathbb{E}_{\bssigma}
[e^2_{N,s}(H_{\bsp} \oplus_{\bsp} \bssigma,K_{s,\bsgamma})]}.\]
The following result is the main result of~\cite{HKP14}.

\begin{thm}[{\cite[Theorem~1]{HKP14}}]\label{thmerrsob}
Let $N \ge 2$. We have
\begin{eqnarray}\label{rmsbd}
[\widehat{e}_{N,s}(H_{\bsp},K_{s,\bsgamma})]^2 \le
\frac{1}{N^2}\left[\prod_{j=1}^s \left(1+\gamma_j (\log N)
\frac{p_j^2}{\log p_j}\right)+
\prod_{j=1}^s\left(1+\frac{\gamma_j}{2}\right) \prod_{j=1}^s
\left(1+\frac{\gamma_j p_j}{6}\right)\right].
\end{eqnarray}
In particular, if $\sum_{j = 1}^\infty \gamma_j \frac{p_j^2}{\log
p_j} < \infty$, then for any $\delta >0$ we have
$$\widehat{e}_{N,s}(H_{\bsp},K_{s,\bsgamma}) \ll_{\delta,\bsgamma,\bsp} \frac{1}{N^{1-\delta}},$$
where the implied constant is independent of the dimension $s$.
\end{thm}

The bound \eqref{rmsbd} is, up to $\log$-factors, optimal. For a
further discussion of the result, especially with respect to the
dependence on the dimension $s$ we refer to \cite{HKP14}.
Theorem~\ref{thmerrsob} can also be interpreted in the
``deterministic'' sense that for every fixed $N \ge 2$ there
exists a $\bsp$-adic shift $\bssigma \in [0,1)^s$ such that the
squared worst-case error of the initial $N$ elements of the
corresponding $\bsp$-adically shifted Halton sequence satisfies
the bound \eqref{rmsbd}. The problem with this interpretation is
that the $\bsp$-adic shift has to be chosen from an uncountable
set, namely the $s$-dimensional unit cube. This is a big drawback
if one wants to effectively find good $\bsp$-adic shifts.

It is the aim of this short paper to show that it suffices to choose the $\bsp$-adic shifts, 
which yield an upper bound of the form \eqref{rmsbd}, from a finite set. 
This set of possible candidates has size $N^s$ which is of course huge already 
for moderately large $s$ or $N$. However we also show, that in principle good shifts 
can be found by a component-by-component (CBC) algorithm. This idea is borrowed from the 
construction of good lattice point sets which goes back to Korobov~\cite{Kor_book} 
and to Sloan and Reztsov~\cite{SR}, and which is nowadays used in a multitude of papers. 
With this ``adaptive search'' the search space is only of a size of order $O(sN)$. 

The rest of the paper is structured as follows: In
Section~\ref{sec_aux} we prove some auxiliary results. The
CBC construction of $\bsp$-adic shifts as well as the statement and
proof of the main result of this paper are presented in
Section~\ref{sec_CBC}.

\section{Auxiliary results}\label{sec_aux}

We use the following notation: for $p \in \NN$ and $m \in \NN_0$
let $$\QQ(p^m):=\{a p^{-m} \ : \ a=0,1,\ldots,p^m-1\}.$$

We now show the following lemma.
\begin{lem}\label{lemsimid1}
Let $H_{p,N}$ be the point set consisting of the first $N$
elements of $H_p$ and let $m\in\NN$ be minimal such that $N<p^m$.
Furthermore, let $\sigma_m \in \QQ(p^m)$. Then it is true that
$$e^2_{N,1}(H_{p,N}\oplus_p ^{\simi}\sigma_m,K_{1,\gamma_1})\le p^m \int_{0}^{p^{-m}}
e^2_{N,1}(H_{p,N}\oplus_p(\sigma_m+\delta),K_{1,\gamma_1})\rd\delta.$$
\end{lem}
\begin{proof}
Let $H_{p,N}=\{h_0,h_1,\ldots,h_{N-1}\}$. From \eqref{fo_wc_err}
we obtain
\begin{eqnarray*}
\lefteqn{p^m \int_{0}^{p^{-m}}
e^2_{N,1}(H_{p,N}\oplus_p(\sigma_m+\delta),K_{1,\gamma_1})\rd\delta= \left(1+\frac{\gamma_1}{3}\right)}\\
&&-\frac{2}{N}\sum_{n=0}^{N-1}p^m \int_{0}^{p^{-m}} 1+\frac{\gamma_1}{2}\left(1-(h_n\oplus_p (\sigma_m +\delta))^2\right)\rd \delta\\
&&+\frac{1}{N^2}\sum_{n=0}^{N-1}p^m \int_{0}^{p^{-m}} 1+\gamma_1\left(1-(h_n\oplus_p (\sigma_m +\delta))\right)\rd \delta\\
&&+\frac{1}{N^2}\sum_{\substack{n,k=0\\n\neq k}}^{N-1}p^m
\int_{0}^{p^{-m}} 1+\gamma_1 \min\left\{1-(h_n\oplus_p (\sigma_m
+\delta)),1-(h_k\oplus_p (\sigma_m +\delta))\right\}\rd \delta.
\end{eqnarray*}
For given $n\in\{0,1,\ldots,N-1\}$, let us now analyze the
quantity
$$h_n\oplus_p (\sigma_m +\delta)=\phi_{p}(\phi_p^+ (h_n) +_{\ZZ_p}\phi_p^+ (\sigma_m +\delta)).$$
The base $p$ expansion of $h_n$ is of the form
$h_n=\sum_{r=1}^{m} \frac{h_n^{(r)}}{p^r}$, since $N< p^m$.
Furthermore, the base $p$ expansions of $\sigma_m$ and $\delta$,
respectively, are of the form
$$\sigma_m=\sum_{r=1}^{m} \frac{\sigma^{(r)}}{p^r} \ \ \ \mbox{ and }\ \ \ \delta=\sum_{r=m+1}^\infty \frac{\delta^{(r)}}{p^r},$$
due to the assumptions on $\sigma_m$ and $\delta$. Consequently,
$$\phi_p^+ (h_n)=\sum_{r=1}^{m} h_n^{(r)} p^{r-1}\quad\mbox{and}\quad
  \phi_p^+ (\sigma_m +\delta)=\phi_p^+ (\sigma_m) +_{\ZZ_p} \phi_p^+ (\delta)=
  \sum_{r=1}^{m} \sigma^{(r)} p^{r-1} +_{\ZZ_p} \sum_{r=m+1}^\infty \delta^{(r)} p^{r-1}.$$
Let $$\phi_p^+ (h_n) +_{\ZZ_p}\phi_p^+ (\sigma_m)=
\sum_{r=1}^{m+1} y_r p^{r-1}$$ with $y_r \in \{0,1,\ldots,p-1\}$.
Then we obtain
\begin{eqnarray*}
\phi_p^+ (h_n) +_{\ZZ_p}\phi_p^+ (\sigma_m
+\delta)&=&\sum_{r=1}^{m} y_r p^{r-1} +_{\ZZ_p} y_{m+1} p^m
+_{\ZZ_p}\phi_p^+(\delta),
\end{eqnarray*}
Note that $\sum_{r=1}^{m} y_r p^{r-1}$ is the truncation of the
$p$-adic sum $\phi_p^+ (h_n) +_{\ZZ_p}\phi_p^+ (\sigma_m)$ to the
first $m$ digits. Hence
$$\phi_p\left(\sum_{r=1}^{m} y_r p^{r-1}\right)= h_n \oplus_p^{\simp} \sigma_m.$$

For short we write
$$\xi(h_n,\sigma_m):= \phi_p(y_{m+1} p^m).$$
Note that $\phi_p^+ (\xi(h_n,\sigma_m))= y_{m+1} p^m$. Hence we
can write
$$h_n\oplus_p (\sigma_m +\delta)=\phi_{p}(\phi_p^+ (h_n) +_{\ZZ_p}\phi_p^+ (\sigma_m +\delta))=(h_n \oplus_p^{\simp} \sigma_m) + (\xi(h_n,\sigma_m)\oplus_p\delta).$$
From this we obtain
\begin{eqnarray*}
\lefteqn{p^m \int_{0}^{p^{-m}} 1+\frac{\gamma_1}{2}\left(1-(h_n\oplus_p (\sigma_m +\delta))^2\right)\rd \delta}\\
&=&p^m \int_{0}^{p^{-m}}
1+\frac{\gamma_1}{2}\left(1-((h_n\oplus_p^{\simp} \sigma_m)
+(\xi(h_n,\sigma_m)\oplus_p\delta))^2\right)\rd {\delta}.
\end{eqnarray*}
We now use \cite[Lemma~3]{HKP14}, which states that for any $f\in
L_2 ([0,1])$ and any $y\in[0,1)$, we have
\begin{equation}\label{le3HKP14}
\int_0^1 f(x)\rd x=\int_0^1 f(x\oplus_{p} y)\rd x.
\end{equation}
This yields
\begin{eqnarray*}
\lefteqn{p^m \int_{0}^{p^{-m}} 1+\frac{\gamma_1}{2}\left(1-(h_n\oplus_p (\sigma_m +\delta))^2\right)\rd \delta=}\\
&=&p^m \int_{0}^{p^{-m}} 1+\frac{\gamma_1}{2}\left(1-((h_n\oplus_p^{\simp} \sigma_m) +\delta)^2\right)\rd {\delta}\\
&=&1+\frac{\gamma_1}{2}\left(1-(h_n\oplus_p^{\simp}
\sigma_m)^2\right)-\frac{1}{p^{m}}\frac{\gamma_1}{2}
(h_n\oplus_p^{\simp} \sigma_m) -
\frac{1}{p^{2m}}\frac{\gamma_1}{6}.
\end{eqnarray*}
Furthermore, in a similar fashion,
\begin{eqnarray*}
\lefteqn{p^m \int_{0}^{p^{-m}} 1+\gamma_1\left(1-((h_n\oplus_p (\sigma_m +\delta))\right)\rd \delta=}\\
&=&p^m \int_{0}^{p^{-m}} 1+\gamma_1\left(1-((h_n\oplus_p^{\simp} \sigma_m) +(\xi(h_n,\sigma_m)\oplus_p\delta))\right)\rd\delta\\
&=&p^m \int_{0}^{p^{-m}} 1+\gamma_1\left(1-((h_n\oplus_p^{\simp} \sigma_m) +\delta)\right)\rd \delta\\
&=&-\frac{\gamma_1}{2}\frac{1}{p^m} +1 +\gamma_1 - \gamma_1
(h_n\oplus_p^{\simp} \sigma_m).
\end{eqnarray*}
Finally, let us deal with the expression
\begin{equation}\label{eqminint}
p^m \int_{0}^{p^{-m}} 1+\gamma_1 \min\left\{1-(h_n\oplus_p
(\sigma_m +\delta)),1-(h_k\oplus_p (\sigma_m +\delta))\right\}\rd
\delta
\end{equation}
with $k\neq n$. Note that, as $k\neq n$, we cannot have
$h_n\oplus_p (\sigma_m +\delta)=h_k\oplus_p (\sigma_m +\delta)$.
Suppose that
\begin{equation}\label{eqlower}
h_n\oplus_p (\sigma_m +\delta)<h_k\oplus_p (\sigma_m +\delta).
\end{equation}
Using the notation introduced above, we can
rewrite~\eqref{eqlower} as
$$(h_n\oplus_p^{\simp} \sigma_m) +(\xi(h_n,\sigma_m)\oplus_p\delta)<(h_k\oplus_p^{\simp} \sigma_m) +(\xi(h_k,\sigma_m)\oplus_p\delta).$$
Again, since $k\neq n$, we cannot have
$$(h_n\oplus_p^{\simp} \sigma_m)=(h_k\oplus_p^{\simp} \sigma_m),$$
as this would also imply $\xi(h_n,\sigma_m)=\xi(h_k,\sigma_m)$,
and would so yield a contradiction to~\eqref{eqlower}.
Furthermore, it cannot be the case that
$$(h_n\oplus_p^{\simp} \sigma_m)>(h_k\oplus_p^{\simp} \sigma_m),$$
since $\xi(h_n,\sigma_m),\xi(h_k,\sigma_m)\in [0,p^{-m})$, and so
we would also end up with a contradiction to~\eqref{eqlower}.
Therefore, we see that~\eqref{eqlower} automatically implies
\begin{equation}\label{eqlowersimp}
(h_n\oplus_p^{\simp} \sigma_m)<(h_k\oplus_p^{\simp} \sigma_m).
\end{equation}
Suppose now, on the other hand, that~\eqref{eqlowersimp} holds.
Then, since $\xi(h_n,\sigma_m),\xi(h_k,\sigma_m)\in [0,p^{-m})$,
also~\eqref{eqlower} must hold. We have thus shown
that~\eqref{eqlower} and~\eqref{eqlowersimp} are equivalent.

Suppose now in the analysis of~\eqref{eqminint}
that~\eqref{eqlower} holds, i.e.,
\begin{eqnarray*}
 \lefteqn{p^m \int_{0}^{p^{-m}} 1+\gamma_1
\min\left\{1-(h_n\oplus_p (\sigma_m +\delta)),1-(h_k\oplus_p (\sigma_m +\delta))\right\}\rd \delta=}\\
&=&p^m \int_{0}^{p^{-m}} 1+\gamma_1
\left(1-(h_k\oplus_p (\sigma_m +\delta))\right)\rd \delta\\
&=&p^m \int_{0}^{p^{-m}} 1+\gamma_1 \left(1-((h_k\oplus_p^{\simp}
\sigma_m) +(\xi(h_k,\sigma_m)\oplus_p\delta))\right)\rd \delta.
\end{eqnarray*}
Using the equivalence between~\eqref{eqlower}
and~\eqref{eqlowersimp}, and again \eqref{le3HKP14}, we see that
the latter expression equals
\begin{eqnarray*}
 \lefteqn{p^m \int_{0}^{p^{-m}} 1+\gamma_1\left(
\min\left\{1-(h_n\oplus_p^{\simp} \sigma_m),1-(h_k\oplus_p^{\simp} \sigma_m)\right\}-(\xi(h_k,\sigma_m)\oplus_p\delta)\right)\rd \delta}\\
&=&p^m \int_{0}^{p^{-m}} 1+\gamma_1
\left(\min\left\{1-(h_n\oplus_p^{\simp} \sigma_m),1-(h_k\oplus_p^{\simp} \sigma_m)\right\}-\delta\right)\rd \delta\\
&=&-\frac{\gamma_1}{2}\frac{1}{p^m} +1+\gamma_1
\min\left\{1-(h_n\oplus_p^{\simp} \sigma_m),1-(h_k\oplus_p^{\simp}
\sigma_m)\right\}.
\end{eqnarray*}
A similar argument holds if the converse of~\eqref{eqlower} holds.

Putting all of these observations together, we obtain
\begin{eqnarray*}
\lefteqn{p^m \int_{0}^{p^{-m}}
e^2_{N,1}(H_{p,N}\oplus_p(\sigma_m+\delta),K_{1,\gamma_1})\rd\delta= \left(1+\frac{\gamma_1}{3}\right)}\\
&&-\frac{2}{N}\sum_{n=0}^{N-1}
\left(1+\frac{\gamma_1}{2}\left(1-(h_n\oplus_p^{\simp}
\sigma_m)^2\right)-\frac{\gamma_1}{2}\frac{1}{p^{m}}
(h_n\oplus_p^{\simp} \sigma_m) - \frac{1}{p^{2m}}\frac{\gamma_1}{6}\right)\\
&&+\frac{1}{N^2}\sum_{n=0}^{N-1}\left(-\frac{\gamma_1}{2}\frac{1}{p^m} +1 +\gamma_1 - \gamma_1 (h_n\oplus_p^{\simp} \sigma_m)\right)\\
&&+\frac{1}{N^2}\sum_{\substack{n,k=0\\n\neq k}}^{N-1}
\left(-\frac{\gamma_1}{2}\frac{1}{p^m} +1+\gamma_1 \min\left\{1-(h_n\oplus_p^{\simp} \sigma_m),1-(h_k\oplus_p^{\simp} \sigma_m)\right\}\right)\\
&\ge&\left(1+\frac{\gamma_1}{3}\right)\\
&&-\frac{2}{N}\sum_{n=0}^{N-1}
\left(1+\frac{\gamma_1}{2}\left(1-(h_n\oplus_p^{\simp}
\sigma_m)^2\right)-\frac{\gamma_1}{2}\frac{1}{2p^{m}}
2(h_n\oplus_p^{\simp} \sigma_m) - \frac{\gamma_1}{2}\frac{1}{4p^{2m}}\right)\\
&&+\frac{1}{N^2}\sum_{n=0}^{N-1}\left(1 +\gamma_1 \left(1-\left(h_n\oplus_p^{\simp} \sigma_m+\frac{1}{2p^m}\right)\right)\right)\\
&&+\frac{1}{N^2}\sum_{\substack{n,k=0\\n\neq k}}^{N-1}
\left(1+\gamma_1 \min\left\{1-\left(h_n\oplus_p^{\simp} \sigma_m+\frac{1}{2p^m}\right),1-\left(h_k\oplus_p^{\simp} \sigma_m +\frac{1}{2p^m}\right)\right\}\right)\\
&=&\left(1+\frac{\gamma_1}{3}\right)-\frac{2}{N}\sum_{n=0}^{N-1}\left(1+\frac{\gamma_1}{2}\left(1-\left(h_n\oplus_p^{\simi} \sigma_m\right)^2\right)\right)\\
&&+\frac{1}{N^2}\sum_{n,k=0}^{N-1}
\left(1+\gamma_1 \min\left\{1-\left(h_n\oplus_p^{\simi} \sigma_m\right),1-\left(h_k\oplus_p^{\simi} \sigma_m\right)\right\}\right)\\
&=&e_{N,1}^2(H_{p,N}\oplus_p^{\simi}\sigma_m,K_{1,\gamma_1}).
\end{eqnarray*}
The result follows.
\end{proof}

For two point sets $X=\{\bsx_0,\bsx_1,\ldots,\bsx_{N-1}\}$ in
$[0,1)^{s_1}$ and $Y=\{\bsy_0,\bsy_1,\ldots,\bsy_{N-1}\}$ in
$[0,1)^{s_2}$ we write $(X,Y)$ to denote the point set consisting
of the concatenated points
$(\bsx_k,\bsy_k)=(x_{k,1},\ldots,x_{k,s_1},y_{k,1},\ldots,y_{k,s_2})$
for $k=0,1,\ldots,N-1$.

\begin{lem}\label{lemsimids}
Let $P_{s,N}$ be a point set of $N$ points in $[0,1)^s$. Let
$H_{p,N}$ be as in Lemma~\ref{lemsimid1} and let $m\in\NN$ be
minimal such that $N<p^m$. Furthermore, let $\sigma_m
\in\QQ(p^m)$. Then it is true that
$$e^2_{N,s+1}((P_{s,N},H_{p,N}\oplus_p ^{\simi}\sigma_m),K_{s+1,\bsgamma})\le p^m \int_{0}^{p^{-m}}
e^2_{N,s+1}((P_{s,N},H_{p,N}\oplus_p(\sigma_m+\delta)),K_{s+1,\bsgamma})\rd\delta.$$
\end{lem}
\begin{proof}
The proof is similar to that of Lemma~\ref{lemsimid1}.
\end{proof}

\section{The CBC construction}\label{sec_CBC}

In this section, we analyze the following CBC construction of a
mid-simplified $\bsp$-adic shift to obtain $\bsp$-adically shifted
Halton sequences with a low integration error.

Throughout this section, let $s,N\in\NN$ be given and let
$\bsp=(p_1,\ldots,p_s)\in \PP^s$ with pairwise distinct components
$p_j$. For $j\in [s]$ let $m_j\in\NN$ be minimal such that
$N<p_j^{m_j}$. Let $H_{\bsp,N}$ be the point set consisting of the
first $N$ elements of $H_{\bsp}$. To stress the dependence of the
worst-case error on the $\bsp$-adic shift we write in the
following $$e_{N,s}(\bssigma):=e_{N,s}(H_{\bsp,N}
\oplus_{\bsp}^{\simi} \bssigma,K_{s,\bsgamma})$$ for $\bssigma \in
\QQ(p_1^{m_1})\times \cdots \times \QQ(p_s^{m_s})$.

We propose the following algorithm.
\begin{algorithm}\label{algcbc}
\begin{itemize}
  \item[(1)] Choose $\sigma_{1}\in\QQ(p_1^{m_1})$ to minimize
$e_{N,1}^2 (\sigma)$ as a function of $\sigma$. \item[(2)] For
$1\le d\le s-1$, assume that $\sigma_{1},\ldots,\sigma_{d}$ have
already been found. Choose $\sigma_{d+1}\in\QQ(p_{d+1}^{m_{d+1}})$
to minimize
\begin{equation}\label{CBC_err}
e_{N,d+1}^2 ((\sigma_1,\ldots,\sigma_d,\sigma))
\end{equation}
as a function of $\sigma$. \item[(3)] If $d \le s-1$ increase $d$
by 1 and go to Step 2, otherwise stop.
\end{itemize}
\end{algorithm}
\begin{rem}\rm
We remark that Algorithm~\ref{algcbc} makes the main result in~\cite{HKP14} much more explicit, 
as the algorithm only needs to check a countable number of possible candidates for the $\bsp$-adic shift.
A slight drawback of our method is that
the effective CBC construction of good $\bsp$-shifts has
a cost of $O(s^2N^3)$ operations, which is still large. Further improvements with respect 
to the construction cost are a demanding problem for future research.\\
\end{rem}

The following theorem states that Algorithm~\ref{algcbc} yields
$\bsp$-adically shifted Halton sequences with a low integration
error. Note that the error bound is of the same order as the one in Theorem~\ref{thmerrsob}.
\begin{thm}\label{thmcbc}
Let the notation be as above, and let $d\in
[s]$. Assume that $\bssigma_s=(\sigma_{1},\ldots,\sigma_{s})$ has
been constructed according to Algorithm~\ref{algcbc}. Let
$\bssigma_d:=(\sigma_1,\ldots,\sigma_d)$. Then
\begin{equation}\label{resultcbc}
 e_{N,d}^2 (\bssigma_d) \le \frac{1}{N^2}
\left(\prod_{j=1}^d \left(1+2\gamma_j (\log N) \frac{p_j^2}{\log
p_j}\right)+\prod_{j=1}^{d}(1+\gamma_j) \prod_{j=1}^d
\left(1+\frac{\gamma_j p_j}{6}\right)\right).
\end{equation}
\end{thm}
\begin{proof}
 We show the result by induction on $d$. For $d=1$ we have
\begin{eqnarray*}
\lefteqn{\int_0^1 e_{N,1}^2
(H_{p_1,N}\oplus_{p_1}\sigma,K_{1,\gamma_1})\rd\sigma}\\ & = &
\frac{1}{p_1^{m_1}}
\sum_{\ell=0}^{p_1^{m_1}-1}p_1^{m_1}\int_{\ell/p_1^{m_1}}^{(\ell+1)/p_1^{m_1}}
e_{N,1}^2 \left(H_{p_1,N}\oplus_{p_1}\left(\frac{\ell}{p_1^{m_1}}+\delta\right),K_{1,\gamma_1}\right)\rd\delta\\
& \ge & \frac{1}{p_1^{m_1}} \sum_{\ell=0}^{p_1^{m_1}-1} e_{N,1}^2
\left(\frac{\ell}{p_1^{m_1}}\right),
\end{eqnarray*}
where we applied Lemma~\ref{lemsimid1}. Hence there exists a
$\sigma'_1\in \QQ(p_1^{m_1})$ such that
\begin{eqnarray*}
 e_{N,1}^2 (\sigma'_1)&\le& \int_0^1 e_{N,1}^2 (H_{p_1,N}\oplus_{p_1}\sigma,K_{1,\gamma_1})\rd\sigma\\
&\le& \frac{1}{N^2}\left(1+2\gamma_1 (\log N) \frac{p_1^2}{\log
p_1}\right)+(1+\gamma_1)\left(1+\frac{\gamma_1 p_1}{6}\right),
\end{eqnarray*}
where we used \cite[Theorem~1]{HKP14} for the second inequality.
Since $\sigma_{1}$ is chosen by Algorithm~\ref{algcbc} to minimize
$e_{N,1}^2 (\sigma)$, it follows that the result holds for $d=1$.

Suppose the result has already been shown for some fixed $d\in
[s-1]$. Assume that $\bssigma_{d}=(\sigma_1,\ldots,\sigma_d)$ has
been obtained by the CBC algorithm. Since $\sigma_{d+1}$ is chosen
in order to minimize the squared error \eqref{CBC_err}, we have
(where we write with some abuse of notation
$(\bssigma_d,\sigma_{d+1}):=(\sigma_1,\ldots,\sigma_d,\sigma_{d+1})$)
$$e_{N,d+1}^2 ((\bssigma_d,\sigma_{d+1}))
\le \frac{1}{p_{d+1}^{m_{d+1}}}
\sum_{v=0}^{p_{d+1}^{m_{d+1}}-1} e_{N,d+1}^2
\left(\left(\bssigma_{d},\frac{v}{p_{d+1}^{m_{d+1}}}\right)\right).$$
Using Lemma~\ref{lemsimids}, we now see that, for any
$v\in\{0,\ldots,p_{d+1}^{m_{d+1}}-1\}$,
\begin{multline*}
e_{N,d+1}^2 \left(\left(\bssigma_{d},\frac{v}{p_{d+1}^{m_{d+1}}}\right)\right)\\
\le p_{d+1}^{m_{d+1}} \int_0^{p_{d+1}^{-m_{d+1}}} e_{N,d+1}^2
\left(\left(H_{\bsp_d,N}\oplus_{\bsp}^{\simi}
\bssigma_{d},H_{p_{d+1},N}\oplus_{p_{d+1}}\left(\frac{v}{p_{d+1}^{m_{d+1}}}+\delta\right)\right),K_{d+1,\bsgamma}\right)\rd\delta,
\end{multline*}
where $\bsp_d:=(p_1,\ldots,p_d)$, and hence
$$
e_{N,d+1}^2 ((\bssigma_d,\sigma_{d+1}))\le \int_0^1  e_{N,d+1}^2
((H_{\bsp_d,N}\oplus_{\bsp}^{\simi}
\bssigma_{d},H_{p_{d+1},N}\oplus_{p_{d+1}}\sigma),K_{d+1,\bsgamma})
\rd\sigma.
$$
We denote the points of $H_{\bsp_d,N}\oplus_{\bsp}^{\simi}
\bssigma_{d}$ by $\bsx_n=(x_{n,1},\ldots,x_{n,d})$, and the points
of $H_{p_{d+1},N}$ by $h_{n}$. Due to \eqref{fo_wc_err}, we obtain
\begin{eqnarray*}
 \lefteqn{\int_0^1  e_{N,d+1}^2 ((H_{\bsp_d,N}\oplus_{\bsp}^{\simi} \bssigma_{\bsm},H_{p_{d+1},N}\oplus_{p_{d+1}}\sigma),K_{d+1,\bsgamma}) \rd\sigma=
\prod_{j=1}^{d+1}\left(1+\frac{\gamma_j}{3}\right)}\\
&&-\frac{2}{N}\sum_{n=0}^{N-1}\left[\prod_{j=1}^d
\left(1+\frac{\gamma_j}{2}(1-x_{n,j}^2)\right)\right]
\int_0^1 \left(1+\frac{\gamma_{d+1}}{2}(1-(h_n\oplus_{p_{d+1}} \sigma)^2)\right)\rd\sigma\\
&&+\frac{1}{N^2}\sum_{n,k=0}^{N-1}\left[\prod_{j=1}^d \left(1+\gamma_j \min\{1-x_{n,j},1-x_{k,j}\}\right)\right]\\
&&\hspace{1cm}\times\int_0^1
\left(1+\gamma_{d+1}\min\{1-(h_n\oplus_{p_{d+1}}\sigma),1-(h_k\oplus_{p_{d+1}}\sigma)\}\right)\rd\sigma.
\end{eqnarray*}
Let now
$$I_1:=\int_0^1 \left(1+\frac{\gamma_{d+1}}{2}(1-(h_n\oplus_{p_{d+1}} \sigma)^2)\right)\rd\sigma,$$
and
$$I_2:=\int_0^1 \left(1+\gamma_{d+1}\min\{1-(h_n\oplus_{p_{d+1}}\sigma),1-(h_k\oplus_{p_{d+1}}\sigma)\}\right)\rd\sigma.$$

Using \eqref{le3HKP14}, we obtain
$$I_1=\int_0^1 \left(1+\frac{\gamma_{d+1}}{2}(1-\sigma^2)\right)\rd\sigma =1+\frac{\gamma_{d+1}}{3}.$$

Let us now deal with $I_2$. 
$$I_2=\sum_{\ell=0}^\infty r_{p_{d+1},\gamma_{d+1}}(\ell)\beta_{\ell}(h_n)\overline{\beta_{\ell} (h_k)},$$
where for $\ell=\ell_{a-1}p_{d+1}^{a-1}+\cdots + \ell_1
p_{d+1}+\ell_0$ with $\ell_{a-1}\neq 0$ we have
$$
r_{p_{d+1},\gamma_{d+1}}=\begin{cases}
                          1+\frac{\gamma_{d+1}}{3} & \mbox{if $\ell=0$,}\\
              \frac{\gamma_{d+1}}{2p_{d+1}^a}\left(\frac{1}{\sin^2(\ell_{a-1}\pi/p_{d+1})}-\frac{1}{3}\right) & \mbox{if $\ell\neq 0$.}
                         \end{cases}
$$
Altogether, we obtain
\begin{eqnarray}\label{eqdd+1}
\lefteqn{e_{N,d+1}^2 ((\bssigma_d,\sigma_{d+1}))}\nonumber\\
&\le&\prod_{j=1}^{d+1} \left(1+\frac{\gamma_j}{3}\right)
-\frac{2}{N}\sum_{n=0}^{N-1}\left[\prod_{j=1}^d \left(1+\frac{\gamma_j}{2}(1-x_{n,j}^2)\right)\right]\left(1+\frac{\gamma_{d+1}}{3}\right)\nonumber\\
&&+\frac{1}{N^2}\sum_{n,k=0}^{N-1}\left[\prod_{j=1}^d
\left(1+\gamma_j \min\{1-x_{n,j},1-x_{k,j}\}\right)\right]
\sum_{\ell=0}^\infty r_{p_{d+1},\gamma_{d+1}}(\ell)\beta_{\ell}(h_n)\overline{\beta_{\ell} (h_k)}\nonumber\\
&=&\left(1+\frac{\gamma_{d+1}}{3}\right) \left[\prod_{j=1}^d \left(1+\frac{\gamma_j}{3}\right)-\frac{2}{N}\sum_{n=0}^{N-1}\prod_{j=1}^d \left(1+\frac{\gamma_j}{2}(1-x_{n,j}^2)\right)\right.\nonumber\\
&& \hspace{2cm}\left. +\frac{1}{N^2}\sum_{n,k=0}^{N-1}\prod_{j=1}^d \left(1+\gamma_j \min\{1-x_{n,j},1-x_{k,j}\}\right)\right]\nonumber\\
&&+\frac{1}{N^2}\sum_{n,k=0}^{N-1}\left(\prod_{j=1}^d
\left(1+\gamma_j \min\{1-x_{n,j},1-x_{k,j}\}\right)\right)
\sum_{\ell=1}^\infty r_{p_{d+1},\gamma_{d+1}}(\ell)\beta_{\ell}(h_n)\overline{\beta_{\ell} (h_k)}\nonumber\nonumber\\
&=&\left(1+\frac{\gamma_{d+1}}{3}\right)e_{N,d}^2(\bssigma_d)+T,
\end{eqnarray}
where
$$T:=\frac{1}{N^2}\sum_{n,k=0}^{N-1}\left(\prod_{j=1}^d \left(1+\gamma_j \min\{1-x_{n,j},1-x_{k,j}\}\right)\right)
\sum_{\ell=1}^\infty
r_{p_{d+1},\gamma_{d+1}}(\ell)\beta_{\ell}(h_n)\overline{\beta_{\ell}
(h_k)}.$$ Since $\min\{1-x_{n,j},1-x_{k,j}\} \le 1$ we obviously
have
\begin{eqnarray}\label{Tbd}
 T \le \left(\prod_{j=1}^d (1+\gamma_j)\right)\sum_{\ell=1}^\infty r_{p_{d+1},\gamma_{d+1}}(\ell)\abs{\frac{1}{N}\sum_{n=0}^{N-1}\beta_{\ell}(h_n)}^2.
\end{eqnarray}
From the proof of \cite[Theorem 1]{HKP14}, it can easily be
derived that
$$\sum_{\ell=1}^\infty r_{p_{d+1},\gamma_{d+1}}(\ell)\abs{\frac{1}{N}\sum_{n=0}^{N-1}\beta_{\ell}(h_n)}^2 \le 
\frac{1}{N^2}\frac{\gamma_{d+1}g p_{d+1}^2}{2} + \frac{\gamma_{d+1}}{6 p_{d+1}^g} \left(1+\frac{\gamma_{d+1}}{2}\right),$$
for arbitrarily chosen $g\in\NN$. By choosing $g=\lfloor
2\log_{p_{d+1}} N\rfloor$ and inserting into \eqref{Tbd}, we arrive
at
\begin{eqnarray}\label{eqT}
 T&\le& \frac{1}{N^2}\prod_{j=1}^d (1+\gamma_j)
\left(\left(\gamma_{d+1}(\log N) \frac{p_{d+1}^2}{\log p_{d+1}}\right)+\frac{\gamma_{d+1}p_{d+1}}{6} \left(1+\frac{\gamma_{d+1}}{2}\right)\right)\nonumber\\
&\le& \frac{1}{N^2}\left(\left(\gamma_{d+1}(\log N) \frac{p_{d+1}^2}{\log p_{d+1}}\right) \prod_{j=1}^d \left(1+2\gamma_j(\log N)\frac{p_j^2}{\log p_j}\right)\right.\nonumber\\
&&\hspace{1cm}\left.+ \frac{\gamma_{d+1}p_{d+1}}{6}
\prod_{j=1}^{d+1} (1+\gamma_j) \prod_{j=1}^d
\left(1+\frac{\gamma_j p_j}{6}\right)\right).
\end{eqnarray}
On the other hand, we have, using the induction assumption,
\begin{eqnarray}\label{eqinduction}
\lefteqn{\left(1+\frac{\gamma_{d+1}}{3}\right)e_{N,d}^2 (\bssigma_{d})}\nonumber\\
& \le & \left(1+\frac{\gamma_{d+1}}{3}\right)\frac{1}{N^2}
\left(\prod_{j=1}^d \left(1+2\gamma_j (\log N) \frac{p_j^2}{\log
p_j}\right)+\prod_{j=1}^{d}(1+\gamma_j)
\prod_{j=1}^d \left(1+\frac{\gamma_j p_j}{6}\right)\right)\nonumber\\
&\le& \frac{1}{N^2}\left(\left(1+\gamma_{d+1}(\log N) \frac{p_{d+1}^2}{\log p_{d+1}}\right) \prod_{j=1}^d \left(1+2\gamma_j(\log N)\frac{p_j^2}{\log p_j}\right)\right.\nonumber\\
&& \hspace{1cm}\left. + \prod_{j=1}^{d+1} (1+\gamma_j)
\prod_{j=1}^d \left(1+\frac{\gamma_j p_j}{6}\right)\right).
\end{eqnarray}
Combining equations~\eqref{eqT} and~\eqref{eqinduction}, and
inserting into~\eqref{eqdd+1}, we obtain
\begin{eqnarray*}
e_{N,d+1}^2 ((\bssigma_d,\sigma_{d+1}))
 \le \frac{1}{N^2}\left(\prod_{j=1}^{d+1} \left(1+2\gamma_j(\log N)\frac{p_j^2}{\log p_j}\right)+\prod_{j=1}^{d+1} (1+\gamma_j) \prod_{j=1}^{d+1} \left(1+\frac{\gamma_j p_j}{6}\right)\right).
\end{eqnarray*}
This is the result for $d+1$, and the theorem is shown.
\end{proof}

\begin{small}
\noindent\textbf{Authors' addresses:}
\noindent Peter Kritzer, Friedrich Pillichshammer\\
Department of Financial Mathematics and Applied Number Theory\\ 
Johannes Kepler University Linz\\
Altenbergerstr.~69, 4040 Linz, Austria\\
E-mail: \texttt{peter.kritzer@jku.at},
\texttt{friedrich.pillichshammer@jku.at}.
\end{small}

\end{document}